\documentclass[12pt]{amsart}

\theoremstyle{plain}
\newtheorem{theorem}{Theorem}[section]
\newtheorem{proposition}[theorem]{Proposition}
\newtheorem{corollary}[theorem]{Corollary}
\newtheorem{lemma}[theorem]{Lemma}

\newtheorem{observation}[theorem]{Observation}
\theoremstyle{cupdefn}
\newtheorem{definition}[theorem]{Definition}

\numberwithin{equation}{section}

\newcommand{\N}{\mathbb{N}}
\newcommand{\Z}{\mathbb{Z}}
\newcommand{\F}{\mathbb{F}}

\usepackage{url}
\usepackage{amsfonts}
\usepackage{amscd}
\usepackage[latin2]{inputenc}
\usepackage{t1enc}
\usepackage[mathscr]{eucal}
\usepackage{indentfirst}
\usepackage{graphicx}
\usepackage{graphics}
\usepackage{pict2e}
\usepackage{epic}
\numberwithin{equation}{section}
\usepackage[margin=2.9cm]{geometry}
\usepackage{epstopdf} 
\usepackage{amsmath}
\usepackage{amsthm}
\usepackage{amssymb}
\usepackage{mathrsfs}
\usepackage{graphicx}
\usepackage{enumerate}
\usepackage{tikz}
\usetikzlibrary{shapes.misc}
\usepackage{float}
\tikzset{cross/.style={cross out, draw=black, minimum size=2*(#1-\pgflinewidth), inner sep=0pt, outer sep=0pt},
	cross/.default={2pt}}

\begin{document}

	\title[Sylow classes of reflection subgroups]{Classification of Sylow classes of parabolic and reflection subgroups in unitary reflection groups}

\author[Kane Townsend]{Kane Douglas Townsend}

\address{The University of Sydney \\ School of Mathematics and Statistics}

\email{kane.d.townsend@gmail.com}

\subjclass{primary 20F55 ; secondary 20D20, 20C33}

\keywords{Unitary reflection groups, reflection subgroups, Sylow subgroups}
	
\begin{abstract}
	Let $\ell$ be a prime divisor of the order of a finite unitary reflection group. We classify up to conjugacy the parabolic and reflection subgroups that are minimal with respect to inclusion, subject to containing an $\ell$-Sylow subgroup. The classification assists in describing the $\ell$-Sylow subgroups of unitary reflection groups up to group isomorphism. This classification also relates to the modular representation theory of finite groups of Lie type. We observe that unless a parabolic subgroup minimally containing an $\ell$-Sylow subgroup is $G$ itself, the reflection subgroup within the parabolic minimally containing an $\ell$-Sylow subgroup is the whole parabolic subgroup.
\end{abstract}

\maketitle

\section{Introduction}\label{intro}
Throughout let $G$ be a finite unitary reflection group acting on the complex vector space $V$ of dimension $n$. For each prime $\ell$ dividing the order of $G$, we classify the parabolic and reflection subgroups up to $G$-conjugacy that are minimal with respect to inclusion, subject to containing a $\ell$-Sylow subgroup. This classification has been done previously in \cite{ME} for finite real reflection groups via an algorithm based on the Borel-de Siebenthal algorithm \cite{BS}. It is clear that a parabolic/reflection subgroup minimally contains an $\ell$-Sylow subgroup of $G$ if and only if its order has the same $\ell$-adic valuation as $\vert G \vert$ and none of its parabolic/reflection subgroups have this property. 

It is sufficient to consider irreducible $G$ since the parabolic/reflection subgroups of $G=G_1 \times ... \times G_k$ are of the form $H=H_1 \times... \times H_k$ where each $H_i$ is a parabolic/reflection subgroup of $G_i$. The classification of irreducible unitary reflection groups was given by Shephard and Todd in \cite{TODD}, which can be found in Lehrer and Taylor \cite[Chap. 8]{GUSTAY}. In \cite{TAYLOR} by Taylor, a classification of parabolic and reflection subgroups of unitary reflection groups is given up to conjugacy. Our classification will follow by a direct proof for the infinite family of reflection groups $G(m,p,n)$ and case by case computations for the $34$ exceptional cases.

Recall a \emph{reflection subgroup} of $G$ is a subgroup generated by reflections and a \emph{parabolic subgroup} is the pointwise stabiliser $G_U$ of some subspace $U$ of $V$. By a theorem of Steinberg \cite[Thm. 1.5]{STEINBERG}, a parabolic subgroup is a reflection subgroup. The conjugacy class of parabolic subgroups minimally containing $\ell$-Sylow subgroups is unique since the class of parabolic subgroups is closed under conjugation and intersection. However, since the class of reflection subgroups is not closed under intersection, we do not necessarily have uniqueness of the conjugacy class of reflection subgroups minimally containing $\ell$-Sylow subgroups.

\begin{definition}
	Call the conjugacy class of parabolic subgroups minimally containing the $\ell$-Sylow subgroups the \emph{$\ell$-Sylow conjugacy class of parabolic subgroups}. Refer to such a parabolic subgroup in the conjugacy class as $P_\ell$. In the case that $G=P_\ell$ we say that $\ell$ is \emph{cuspidal} for $G$.
\end{definition}

\begin{definition}
	Call a conjugacy class of reflection subgroups minimally containing the $\ell$-Sylow subgroups a \emph{$\ell$-Sylow conjugacy classes of reflection subgroups}. Refer to such a reflection subgroup in one of these conjugacy classes as $R_\ell$. If $G=R_\ell$ we say that $\ell$ is \emph{supercuspidal} for $G$.
\end{definition}

We will first classify the $P_\ell$ up to $G$-conjugacy. It is then sufficient to consider only the cuspidal cases while classifying the $R_\ell$ up to $G$-conjugacy, since the parabolic closure of a $R_\ell$ is a $P_\ell$ by \cite[Cor. 2.5]{ME}.

The motivations for studying these minimal classes containing $\ell$-Sylow subgroups are various. A major application is in its assistance in describing the $\ell$-Sylow subgroups of unitary reflection groups. It assists since the $\ell$-Sylow subgroup of some irreducible $G$ is the $\ell$-Sylow subgroup of some supercuspidal reflection group up to group isomorphism. We describe the $\ell$-Sylow subgroups of the supercuspidal cases in Section \textsection\ref{observe}, with the supercuspidal cases summarised in Table \ref{table:sylowcases}. This gives a complete description of $\ell$-Sylow subgroups of unitary reflection groups up to group isomorphism.

Another application is in the modular representation theory of finite reductive groups. In particular, \cite[Thm. 4.2]{GHM} states the following:

\begin{theorem}[Geck-Hisse-Malle]
	Let $\mathbf{G}$ be a connected reductive group with connected centre and some $\F_q$-rational structure where $q$ is a power of a prime different from $\ell$. Let $F$ be the Frobenius morphism of $\mathbf{G}$ and $T\subset B$ be an $F$-stable torus and Borel subgroup respectively. Define $\mathcal{L}=\{(L^F_J)^n \mid J\subseteq S, n \in N^F\}$ to be the $N$-conjugates of standard Levi subgroups of $\mathbf{G}^F$, where $S$ is the standard generators of the Weyl group $W^F$ of $\mathbf{G}^F$.
	\item[(i)] If $\ell$ divides $[G:L^F]$ for all $L^F\in\mathcal{L}$$\setminus\{G^F\}$, then the $\ell$-modular Steinberg character is cuspidal.
	\item[(ii)] Let $L^F\in\mathcal{L}\setminus\{G^F\}$ such that $\ell$ does not divide $[G^F:L^F]$. If $D\in\text{Syl}_\ell (L^F)$ satisfies $C_{\mathbf{G^F}}(D) \leq L^F$, then the semisimple vertex of the $\ell$-modular Steinberg character of $\mathbf{G}^F$ is contained in $L^F$.
\end{theorem}

This theorem raises the question of finding the minimal Levi subgroups containing $\ell$-Sylow subgroups in finite reductive groups. In the case of Weyl groups, the classification of Sylow classes of parabolic subgroups answers a special case of this question, as will be seen in a forthcoming paper. Furthermore, in Enguehard and Michel \cite[Thm. 3.2]{MICHEL} it is seen that the $\ell$-Sylow subgroups of a finite reductive group depend partly on the $\ell$-Sylow subgroups of a unitary reflection group.

In \textsection\ref{prenotes} we introduce the basic definitions of unitary reflection groups and notation regarding the classification of reflection subgroups as seen in \cite{TAYLOR}. In \textsection \ref{paraclassify} we classify the $\ell$-Sylow conjugacy class of parabolic subgroups. In \textsection \ref{reflectclassify} we classify the $\ell$-Sylow conjugacy classes of reflection subgroups for the cuspidal cases, allowing us to deduce $R_\ell$ in all cases. In \textsection \ref{observe} we use our classification of the supercuspidal cases seen in Table \ref{table:sylowcases} to describe the $\ell$-Sylow subgroups of unitary reflection groups. In the tables of \textsection \ref{tables} we present the classification of $\ell$-Sylow class of parabolic subgroups, the cuspidal cases, the $\ell$-Sylow classes of reflection subgroups, the supercuspidal cases and the cases where the $\ell$-Sylow classes of reflection subgroups are not unique. By inspection of these tables we are able to note the following observation.

\begin{observation}
	If $\ell$ is not cuspidal for an irreducible $G$, it is supercuspidal for $P_\ell$.
\end{observation}

The next proposition can help to quickly observe that $\ell$ is cuspidal for $G$ in some cases. It is used in Section \ref{paraclassify} for the classifications of $\ell$-Sylow class of parabolic subgroups for the exceptional cases. The proposition involves the characterisation of a unitary reflection group $G$ in terms of its ring of invariants being a polynomial algebra, where the product of the degrees of the chosen algebraically independent homogeneous generating polynomials is the order of $G$ (see \cite[Chap. 3-4]{GUSTAY}).

\begin{proposition}\label{degrees}
	Let $G$ be an irreducible unitary reflection group on $V$ and $d_1,...,d_n$ be the degrees of the homogeneous algebraically independent polynomials generating the $G$-invariant polynomial algebra. If $\ell \mid d_i$ for all $1\leq i \leq n$ then $\ell$ is cuspidal for $G$.
\end{proposition}

\begin{proof}
	By \cite[Corollary 3.24]{GUSTAY} the centre $Z(G)$ has order $\text{gcd}(d_1,...,d_n)$. Then by Schur's Lemma $Z(G)$ acts on $V$ as scalar multiplication, so is a cyclic group acting on $V$ as the $\text{gcd}(d_1,...,d_n)^{\text{th}}$ roots of unity. Let $P$ be a proper parabolic subgroup of $G$, hence fixing some nonzero subspace $U$ of $V$. Then $P \cap Z(G)=\{1\}$, as this is the only element of $Z(G)$ that fixes $U$ pointwise. Therefore, $P$ cannot contain a $\ell$-Sylow subgroup if $\vert Z(G)\vert=\text{gcd}(d_1,...,d_n)$ is divisible by $\ell$.
\end{proof}

We can also observe from the tables which cases that the $\ell$-Sylow class of reflection subgroups is not unique. They are collected in Table \ref{table:notunique} for clarity. It is interesting to note that in the real reflection group case the non-unique classes were made of isomorphic groups generated by dual root systems as seen in \cite{ME}. However, extending to the unitary reflection groups, there are non-unique classes that are not made of isomorphic groups as seen in the cases $G_9$ with $\ell=3$, $G_{17}$ with $\ell=3$, $G_{18}$ with $\ell=2$, $G_{21}$ with $\ell=5$ and $G_{26}=M_3$ with $\ell=3$.

\vspace{12pt}
\section{Notation and preliminaries}\label{prenotes}

A \emph{complex reflection} is a linear transformation of $V$ with finite order and fixed space is a hyperplane. A \emph{unitary reflection group} or \emph{finite complex reflection group} $G$ on $V$ is a finite group generated by complex reflections. The name unitary reflection group comes from the fact that every finite subgroup of $GL(V)$ preserves a positive definite hermitian form $(-,-)$ on $V$ and so the group is unitary with respect to this hermitian form. 

The imprimitive unitary reflection group $G(m,p,n)$ introduced by Shephard and Todd is defined in \cite[Section 3]{TAYLOR} in the following way. Let $[n]:=\{1,2,...,n\}$ and $\{e_i \ \vert \ i\in [n]\}$ be an orthonormal basis for the $n$ dimensional complex vector space $V$. Let $\mu_m$ be the group of $m$th roots of unity. We write $\hat{\theta}$ as the linear transformation that maps $e_i$ to $\theta(i)e_i$ for each $1\leq i\leq n$, where $\theta: [n]\to\mu_m$. Let $p\mid m$, define $A(m,p,n)$ to be the group of all linear transformations $\hat{\theta}$ such that $\prod_{i=1}^{n}\theta(i)^{\frac{m}{p}}=1$. We also define the action of $\pi\in \text{Sym}(n)$ on $V$ to be $\pi(e_i)=e_{\pi(i)}$. The group $G(m,p,n)$ is the semidirect product of $A(m,p,n)$ by Sym$(n)$. Shephard and Todd \cite{TODD} proved every irreducible unitary reflection group belongs to a list $G_k$ for $1\leq k\leq 37$. We call the $k$ the \emph{Shephard-Todd number} of the irreducible unitary reflection group. The Shephard-Todd numbering can be found in \cite{GUSTAY}. For example, $G_1$ is $G(1,1,n)\cong Sym(n)$ with $n\geq2$, $G_2$ is $G(m,p,n)$ with $m>1$ and $n>1$, $G_3$ is $G(m,1,1)\cong \mathcal{C}_m$ with $m>1$, $G_4-G_{22}$ are the primitive rank two groups and $G_{23}-G_{37}$ are the primitive groups of rank greater than two.

We now introduce some notation seen in \cite{TAYLOR}. This is necessary to understand the classifications of parabolic subgroups and reflection subgroups of $G(m,p,n)$ up to conjugacy.
\begin{definition}
	Call $(m',p',n')$ a \emph{feasible triple} for $G(m,p,n)$ if $m,p$ and $n$ are positive integers and $n'\leq n, p'\mid m', m'\mid m$ and $\frac{m'}{p'}\mid\frac{m}{p}$.
\end{definition}

By \cite[Lemma 3.3]{TAYLOR} we have $(m',p',n')$ is a feasible triple if and only if $G(m',p',n')$ arises as a reflection subgroup of $G(m,p,n)$. We define a total order on the feasible triples by writing $(m_1,p_1,n_1)\geq(m_2,p_2,n_2)$ if $n_1>n_2$; or $n_1=n_2$ and $m_1>m_2$; or $n_1=n_2$, $m_1=m_2$ and $p_1\geq p_2$.

\begin{definition}
	A \emph{partition} $\lambda$ of $n$, written $\lambda\vdash n$, is a sequence $\lambda=(n_1,n_2,...,n_d)$ where $n_1, n_2,...,n_d$ are positive integers such that $n=\sum_{i=1}^{d}n_i$ and $n_1\geq n_2\geq...\geq n_d$.
\end{definition}

\begin{definition}
	An \emph{augmented partition} for $G(m,p,n)$ is a decreasing sequence $\Delta=[\tau_1,\tau_2,...,\tau_d]$ of feasible triples $\tau_i=(m_i,p_i,n_i)$ such that $\lambda=(n_1,n_2,...,n_d)\vdash n$.
\end{definition}

Let $k_0=0$ and for $1\leq i \leq d$ let $k_i=n_1+n_2+...+n_i$, then set $\Lambda_i=\{e_j \ \vert \ k_{i-1}< j\leq k_i \}$. We say $(\Lambda_1,\Lambda_2,...,\Lambda_d)$ is the \emph{standard partition} of $\Lambda=\{e_1,...,e_n\}$ associated with $\Delta$. The standard reflection subgroup of type $\Delta$ is $$G_\Delta = \prod_{i=1}^{d} G(m_i,p_i,n_i)$$ where $G(m_i,p_i,n_i)$ acts on the subspace of $V$ with basis $\Lambda_i$. For an $\alpha \in\mu_m$ define $\theta_\alpha:[n]\to\mu_m$ by $\theta_\alpha(1)=\alpha$ and $\theta_\alpha(i)=1$ for $i>1$. Then define
$$G_\Delta^\alpha=\hat{\theta}_\alpha G_\Delta\hat{\theta}_\alpha^{-1}.$$ 

Furthermore, \cite[Theorem 3.9]{TAYLOR} states that for a given augmented partition $$\Delta=[(m_1,p_1,n_1),(m_2,p_2,n_2),...,(m_d,p_d,n_d)]$$ for $G(m,p,n),$ then for $\alpha,\beta \in\mu_m$, we have $G_\Delta^\alpha$ and $G_{\Delta'}^\beta$ are conjugate in $G(m,p,n)$ if and only if $\Delta=\Delta'$ and $\alpha\beta^{-1}\in\mu_k$, where $$k=\frac{m}{\text{gcd}(p,n_1,n_2,...,n_d,\frac{m}{m_1},\frac{m}{m_2},...,\frac{m}{m_d})}.$$

We now can present the classifications of reflection and parabolic subgroups of $G(m,p,n)$, which are given in \cite[Theorem 3.7]{TAYLOR} and \cite[Theorem 3.11]{TAYLOR} respectively. 

\begin{theorem}\label{reflectionsandparabolics}
	Any reflection subgroup of $G(m,p,n)$ is conjugate to $G_\Delta^\alpha$ for some augmented partition $\Delta$ and $\alpha\in \mu_m$. The parabolic subgroups of $G(m,p,n)$ are of the form
	\item[(i)] $G_\Delta= G(m,p,n_j)\times\prod_{i\neq j} G(1,1,n_i)$, where $(n_i)\vdash n$,
	
	\item[(ii)] $G_\Delta^\alpha=\hat{\theta}_\alpha G_\Delta \hat{\theta}_\alpha^{-1},$ where all feasible triples have $m_i=p_i=1$ and $\alpha\in\mu_m$.
\end{theorem}

\vspace{12pt}
\section{Classifying the $\ell$-Sylow  Class of Parabolic Subgroups}\label{paraclassify}

We will now classify the $\ell$-Sylow class of parabolic subgroups for each irreducible $G$. We define the $\ell$-adic valuation of a positive integer as $\nu_\ell(n):=\text{max}\{v\in\N \ | \ \ell^v | n \}$. Since the $\ell$-Sylow class of parabolic subgroups is always unique, it is made up of the smallest order parabolic subgroups whose order have the same $\ell$-adic valuation as $|G|$. Note that we can use Proposition \ref{degrees} to quickly observe cuspidal cases whenever it is applicable.

For the exceptional cases $G_4-G_{37}$ we have the classification of parabolic subgroups of rank greater than two primitive unitary reflection groups directly from \cite{TAYLOR}. For the rank two primitive cases, which are not included in the tables of \cite{TAYLOR}, we ran the MAGMA code mentioned in \cite{TAYLOR} found at \url{http://www.maths.usyd.edu.au/u/don/software.html} to see the parabolic subgroups of these cases. This gives the classification for $G_k$ for $4\leq k\leq 37$ as found in Table \ref{table:parabolic}. 

For $G_1=G(1,1,n)$, we have already done the classification in \cite[Table 1]{ME} as it is the same as the classification for the real reflection group of type $A_{n-1}$. For $G_3=G(m,1,1)$, where $m>1$ the classification is simply the whole group as it has no proper parabolic subgroup other than the trivial group by Theorem \ref{reflectionsandparabolics}. These classifications of $\ell$-Sylow classes of parabolic subgroups are found in Table \ref{table:parabolic}. We now note the following results regarding $\ell$ divisibility of factorials since $\nu_\ell(|G(m,p,n)|)=n\nu_\ell(m)-\nu_\ell(p)+\nu_\ell(n!)$. These results are required for the case $G_1$ done in \cite{ME} and $G_2$ which in one case we will see reduces to the case of $G_1$. The results will also be used in \textsection \ref{reflectclassify}. First we state a form of Kummer's Theorem for multinomial coefficients.

\begin{lemma}\label{KUMMER}
	If $\lambda\vdash n$, then $\nu_\ell(n!)-\sum_{i=0}^{k}\nu_\ell(\lambda_i!)$ is equal to the number of carries when summing $\lambda_i$ in base $\ell$.
\end{lemma}

Furthermore, we have the following Corollary of Lemma \ref{KUMMER} proved in \cite[Lemma 2.6]{ME}.

\begin{corollary}\label{minimisepartitionorder}
	Let the base-$\ell$ expression of $n$ be $(b_rb_{r-1}...b_1b_0)_\ell$. Then the partition $\lambda\vdash n$ that provides the minimum of the set $$\left\{\prod_{i=1}^{k} \lambda_i! \ \bigg| \ \nu_\ell(n!)=\sum_{i=1}^{k}\nu_\ell(\lambda_i!)\right\}$$ is given by the join $\lambda=\lambda^r\cup\lambda^{r-1}\cup...\cup\lambda^{1}\cup\lambda^{0}$ where $\lambda^j=(\ell^j,...,\ell^j)$ with length $b_j$.
\end{corollary}

\begin{theorem}\label{g2parabolic}
	Let $\ell$ be a prime divisor of the order of $G_2=G(m,p,n)$, where $m>1$ and $n>1$. Let the base-$\ell$ expression of $n$ be $(b_kb_{k-1}...b_1b_0)_\ell$. 
	\begin{itemize}
		\item (i) If $\ell\mid m$, then the $\ell$-Sylow class of parabolic subgroups of $G_2$ has a unique element $G(m,p,n)$.
		\item (ii)  If $\ell\nmid m$, then the $\ell$-Sylow class of parabolic subgroups of $G_2$ is the $G_2$-conjugates of	$\prod_{i=1}^{k} G(1,1,\ell^i)^{b_i}$.
	\end{itemize}
\end{theorem}

\begin{proof}
	Suppose $\ell\mid m$. By definition of $A(m,p,n)$, its $\ell$-Sylow subgroups will only fix the zero vector of $V$. Therefore, a parabolic subgroup containing a non-trivial $\ell$-Sylow subgroup of $G(m,p,n)$ will only fix zero, and so we have the result.
	
	Now suppose $\ell \nmid m$.  It is clear that $G(1,1,n)$ contains an $\ell$-Sylow subgroup of $G$, reducing to the case of $G_1$.
\end{proof}

\vspace{12pt}
\section{Classifying the $\ell$-Sylow Class of Reflection Subgroups}\label{reflectclassify}

We classify the $\ell$-Sylow conjugacy classes of reflection subgroups of unitary reflection groups for the cuspidal cases seen in Table \ref{table:cuspidalcases}.  In \cite{TAYLOR} a classification of reflection subgroups is given up to conjugacy. Using the tables in \cite{TAYLOR} we can deduce the classification of $\ell$-Sylow conjugacy classes of reflection subgroups for $G_4-G_{37}$. We again note that the tables for $G_4-G_{22}$ are not included in the paper, but can be found by running the MAGMA code in the link previously given in \textsection \ref{paraclassify}. The  classification of $\ell$-Sylow conjugacy classes of reflection subgroups can be found in Tables \ref{table:reflection} and \ref{table:reflectionrank2}. The $\ell$-Sylow conjugacy classes of reflection subgroups of $G_1$ seen in Table \ref{table:reflection} follows from \cite{ME}. We now classify for the cases $G_2$ and $G_3$.

\begin{theorem}\label{g2reflection}
	Let $\ell$ be a prime divisor of the order of $G_2=G(m,p,n)$, where $m>1$ and $n>1$. Let the base-$\ell$ expression of $n$ be $(b_kb_{k-1}...b_1b_0)_\ell$.
	\begin{itemize}
		\item[(i)] If $\ell \mid p$, then the $\ell$-Sylow classes of reflection subgroups are the conjugacy classes of $\hat{\theta}_\alpha G(\ell^{\nu_\ell(m)},\ell^{\nu_\ell(p)},n) \hat{\theta}_\alpha^{-1}$, where $\alpha\in\{e^{\frac{i\pi k}{m}} \  | \ 0 \leq k\leq \text{gcd}(\frac{p}{\ell^{\nu_\ell(p)}}, n)-1 \}$.
		\item[(ii)] If $\ell \nmid p$, then the $\ell$-Sylow class of reflection subgroups is the unique conjugacy class of $\prod_{i=0}^{k}G(\ell^{\nu_\ell(m)},1,\ell^i)^{b_i}$.
	\end{itemize}
\end{theorem}

\begin{proof}
	By the cuspidal cases seen in Table \ref{table:cuspidalcases}, we only need to consider when $\ell \mid m$. Let $R_\ell$ be such that it is $G_\Delta=\prod_{i=1}^{d}G(m_i,p_i,n_i)$ or $G_\Delta^\alpha$ for some augmented partition $\Delta$ and $\alpha\in\mu_m$. The $m_i$ and $p_i$ must be powers of $\ell$, as otherwise $G_\Delta$ contains a proper reflection subgroup with the same $\ell$-adic valuation giving a contradiction of minimality. Therefore, $G_\Delta$ is a reflection subgroup of $H:=\prod_{i=1}^{d}G(\ell^{\nu_\ell(m)},\ell^{\nu_\ell(p)},n_i)$. By assumption $\nu_\ell(|H|)=\nu_\ell(|G|)$, so we have
	$$n\nu_\ell(m)-d\nu_\ell(p)+\sum_{i=1}^{d}\nu_\ell(n_i!)=n\nu_\ell(m)-\nu_\ell(p)+\nu_\ell(n!).$$
	Then by Kummer's theorem we have $d\nu_\ell(p)\leq \nu_\ell(p)$. Since $d\geq 1$ and $\nu_\ell(p)\geq 0$, we have either $\nu_\ell(p)=0$ or $d=1$.
	
	Hence, we now also assume that $\ell\mid p$, so $d=1$. This gives $G_\Delta=G(m_1,p_1,n)$ a reflection subgroup of $H=G(\ell^{\nu_\ell(m)},\ell^{\nu_\ell(p)},n)$. Now $\nu_\ell(|G_\Delta|)=\nu_\ell(|H|)$ gives 
	$$n\nu_\ell(m_1)-\nu_\ell(p_1)=n\nu_\ell(m)-\nu_\ell(p).$$ Also, since $\frac{m_1}{p_1}\mid\frac{m}{p}$ we have $$\nu_\ell(m_1)-\nu_\ell(p_1)\leq\nu_\ell(m)-\nu_\ell(p).$$ Subtracting this from the above equation gives $\nu_\ell(m_1)\geq\nu_\ell(m)$, since $n>1$. Combining this with $m_1\mid m$ we have $\nu_\ell(m_1)=\nu_\ell(m)$, and then $\nu_\ell(p_1)=\nu_\ell(p)$ follows. So by \cite[Theorem 3.8]{TAYLOR} we know that there are $\text{gcd}(\frac{p}{\ell^{\nu_\ell(p)}},n)$ $\ell$-Sylow classes of reflection subgroups, each class made up of $G_2$-conjugates of $G_\Delta^\alpha=G(\ell^{\nu_\ell(m)},\ell^{\nu_\ell(p)},n)$, where $\alpha\in\{e^{\frac{i\pi k}{m}} \  | \ 0 \leq k\leq \text{gcd}(\frac{p}{\ell^{\nu_\ell(p)}},n)-1 \}$.
	
	We now assume $\ell\nmid p$. Hence, $G_\Delta=\prod_{i=1}^{d} G(m_i,p_i,n_i)$ is a subgroup of $H=\prod_{i=1}^{d}G(\ell^{\nu_\ell(m)},1,n_i)$. If $\nu_\ell(m_i)=\nu_\ell(m)$ for some $i$ then $\nu_\ell(p_i)=0$ so the group orders have equal $\ell$-adic valuations. Now if $\nu_\ell(m_i)=\nu_\ell(m)-x_i,$ for some $x_i\in\Z^+$, then $\nu_\ell(p_i)=-n_ix_i$, giving a contradiction. Hence, $G_\Delta= \prod_{i=1}^{d}G(\ell^{\nu_\ell(m)},1,n_i)$ for some partition $(n_i)\vdash n$. Corollary \ref{minimisepartitionorder} gives us the reflection subgroup $\prod_{i=0}^{k}G(\ell^{\nu_\ell(m)},1,\ell^i)^{b_i}$, which is clearly also minimal with respect to containment. The conjugacy class of groups of this type is unique by \cite[Theorem 3.8]{TAYLOR} and any reflection subgroup preserving the $\ell$-adic valuation of $G_3$ will contain an element in this conjugacy class. Hence, the result follows.
\end{proof}

\begin{theorem}\label{g3reflection}
	Let $\ell$ be a prime divisor of the order of $G_3=G(m,1,1)$, where $m>1$. The $\ell$-Sylow reflection subgroup of $G_3$ is $G(\ell^{\nu_\ell(m)},1,1)$.
\end{theorem}

\begin{proof}
	Trivial.
\end{proof}

\vspace{12pt}
\section{Sylow subgroups of Unitary reflection groups}\label{observe}

Our classification of $\ell$-Sylow classes of reflection subgroups significantly reduces the work required to describe the group isomorphism types of the $\ell$-Sylow subgroups of unitary reflection groups. An $\ell$-Sylow subgroup of a general unitary reflection group is the direct product of particular $\ell$-Sylow subgroups of the irreducible components. Furthermore, an $\ell$-Sylow subgroup of an irreducible unitary reflection group $G$ will be an $\ell$-Sylow subgroup of $R_\ell$. Therefore, to describe the $\ell$-Sylow subgroups of all unitary reflection groups it is sufficient to describe the $\ell$-Sylow subgroups of the supercuspidal cases listed in Table \ref{table:sylowcases}. We will now list the $\ell$-Sylow subgroups up to group isomorphism for all the supercuspidal cases. The description of the isomorphism classes of the unitary reflection groups can be found in \cite[Chap. 6 \S 2-\S4 and Chap. 8 \S 10]{GUSTAY}. We will require the following well known result regarding the $\ell$-Sylow subgroups of the symmetric group that can be found in \cite[Pg. 82]{HALL}.

\begin{proposition}\label{sylowsym}
	The $\ell$-Sylow subgroups of Sym$(n)$ are isomorphic to $\prod_{i=1}^{k} [C_{\ell^i}^{(i)}]^{a_i}$,
	where $n=(a_ka_{k-1}...a_1a_0)_\ell$ and $C_{\ell^i}^{(i)}$ is the iterated wreath product of $i$ copies of the cyclic group of order $\ell^i$.
\end{proposition}

We can now describe the $\ell$-Sylow subgroups up to isomorphism for the supercuspidal cases:

\begin{itemize}
	\item $G(1,1,\ell^i)$ is isomorphic to  $\text{Sym}(\ell^i)$. By Proposition \ref{sylowsym} the $\ell$-Sylow subgroups are isomorphic to  $C_{\ell^i}^{(i)}$. 
	\item $G(\ell^i,\ell^j,n)$ is isomorphic to the semidirect product $A(\ell^i,\ell^j,n)\rtimes \text{Sym}(n)$. By Proposition \ref{sylowsym} the $\ell$-Sylow subgroups are isomorphic to  $A(\ell^i,\ell^j,n)\rtimes \prod_{\zeta=1}^{k}[C_{\ell^\zeta}^{(\zeta)}]^{a_\zeta}$,
	where $n=(a_ka_{k-1}...a_1a_0)_\ell$. 
	\item $G(\ell^i,1,\ell^j)$ is isomorphic to the semidirect product $A(\ell^i,1,\ell^j)\rtimes\text{Sym}(\ell^j)$. By Proposition \ref{sylowsym} the $\ell$-Sylow subgroups are isomorphic to $A(\ell^i,1,\ell^j)\rtimes C_{\ell^j}^{(j)}$.
	\item $G(\ell^i,1,1)$ has $\ell$-Sylow subgroup being the whole group.
	\item $G_4=L_2$ has a unique $2$-Sylow subgroup isomorphic to a quaternion group.
	\item $G_8$ has cyclic groups of order $3$ as its $3$-Sylow subgroups.
	\item $G_{12}$ is isomorphic to $\text{GL}_2(3)$. The $2$-Sylow subgroup is the semidihedral group of order $16$.
	\item $G_{16}$ is isomorphic to the direct product of a cyclic group of order $5$ and the binary icosahedral group. Hence, the $2$-Sylow subgroup of $G_{16}$ is the $2$-Sylow subgroup of binary icosahedral group, which is the quaternion group. 
	\item $G_{16}$ has cyclic groups of order $3$ as its $3$-Sylow subgroups.
	\item $G_{20}$ has cyclic groups of order $5$ as its  $5$-Sylow subgroups.
	\item $G_{24}=J_3^{(4)}$ has cyclic groups of order $7$ as its $7$-Sylow subgroups.
	\item From \cite[Table 3]{GUSTAY} it is clear that the first three standard generators of $L_4$ generate $L_3$ and \cite[Thm. 8.43]{GUSTAY} shows $L_4$ is isomorphic to $C_3\times \text{Sp}_4(3)$. Hence, the natural homomorphism from $L_4$ to its quotient modulo its central subgroup $C_3$ restricts to an embedding of $L_3$ in $\text{Sp}_4(3)$ since the central $C_3$ is not contained in $L_3$.
	
	By comparing orders we see that the  $3$-Sylow subgroups of $G_{25}=L_3$ are isomorphic to the $3$-Sylow subgroups of Sp$_4(3)$. Regard Sp$_4(3)$ as the group preserving the alternating form with matrix $\left( {\begin{array}{cc}
		0 & J \\
		-J & 0 \\
		\end{array} } \right),$ where $J=\left( {\begin{array}{cc}
		0 & 1 \\
		1 & 0 \\
		\end{array} } \right)$. Then the set of upper unitriangular matrices of Sp$_4(3)$ is a $3$-Sylow subgroup. We can describe the $3$-Sylow subgroup as the semidirect product $E\rtimes H$, where $E\cong C_3^3$ is an elementary abelian group and $H\cong C_3$. The group $E$ consists of the matrices $\left( {\begin{array}{cc}
		I & A \\
		0 & I \\
		\end{array} } \right)$ where $A=\left( {\begin{array}{cc}
			a & b \\
			c & a \\
	\end{array} } \right)$ and $H$ consists of the matrices $\left( {\begin{array}{cc}
	D & 0 \\
	0 & D^{-1} \\
	\end{array} } \right)$ where $D=\left( {\begin{array}{cc}
		1& d \\
		0 & 1 \\
\end{array} } \right)$ with $a,b,c,d\in\F_3$.
	 
	\item $G_{32}=L_4$  is isomorphic to $C_3\times\text{Sp}_4(3)$. By definition of the symplectic group, Sp$_4(3)$ has a subgroup isomorphic to $(\text{SL}_2(3)\times\text{SL}_2(3))\rtimes C_2$, where the $C_2$ acts by swapping the factors. Hence the $2$-Sylow subgroup of $L_4$ can be described as $(Q_8\times Q_8)\rtimes C_2$ where $C_2$ acts by swapping the factors and $Q_8$ is the quaternion group.
	\item $G_{32}=L_4$ has cyclic groups of order $5$ as its $5$-Sylow subgroups.
	\item $G_{35}=E_6$ has a subgroup isomorphic to $L_3$ by \cite[Table 4]{BREWER}, which contains a $3$-Sylow subgroup of $E_6$. Hence, the $3$-Sylow subgroup of $E_6$ has the same description as the $3$-Sylow subgroup of $G_{25}=L_3$ given above.
\end{itemize}
\vspace{12pt}
\section{Acknowledgements}
The author would like to thank their PhD supervisor, Anthony Henderson, for many helpful discussions regarding this paper. The author also acknowledges the support from the School of Mathematics and Statistics at The University of Sydney, where the research for this paper took place. The author also thanks the referee for many helpful comments, particularly with the description of the $3$-Sylow subgroup of $G_{25}$.

\newpage
\section{Tables}\label{tables}

We use the same notation as in \cite[Section 5]{TAYLOR} and \cite[Chapter 6]{GUSTAY}. We note \cite[Section 5]{TAYLOR} omits defining $B_n^{(3)} := G(3,1,n)$. 

\begin{table}[H]
	\centering
	\caption{Type of $P_\ell$ in $G$}
	\scalebox{0.7}{
		\begin{tabular}{ | c | c | c | c | c | c |}
			\hline
			{$G$} & $|G|$ & $\ell$ & $P_\ell$ & $|P_\ell|$ & $\ell$ cuspidal \\ \hline \hline
			\begin{tabular}{@{}c@{}}  $G_1=G(1,1,n),$ \\ $n\geq 2$ \end{tabular} & $n!$ & $\ell \mid n!$ & $\displaystyle\prod_{i=1}^{k} G(1,1,\ell^i)^{b_i}$ & $\displaystyle\prod_{i=1}^{k} (\ell^i!)^{b_i}$ & $n=\ell^q$ for $q\geq 2$ \\ \hline
			\begin{tabular}{@{}c@{}} $G_2=G(m,p,n),$  \\ $m>1, n>1$ \end{tabular} & $\frac{m^nn!}{p}$  & $\ell \mid \frac{m^nn!}{p}$  & $\left\{ \begin{array}{ll}
			G(m,p,n) & \text{for} \ \ell\mid m, \\
			\displaystyle\prod_{i=1}^{k} G(1,1,\ell^i)^{b_i} & \text{for} \ \ell \nmid m
			\end{array} \right.$ & $\left\{ \begin{array}{ll}
			\frac{m^nn!}{p} & \text{for} \ \ell\mid m, \\
			\displaystyle\prod_{i=1}^{k} (\ell^i!)^{b_i} & \text{for} \ \ell \nmid m
			\end{array} \right.$ & $\ell\mid m$ \\ \hline
			\begin{tabular}{@{}c@{}}  $G_3=G(m,1,1),$ \\ $m>1$ \end{tabular} & $m$ & $\ell \mid m$ & $G(m,1,1)$ & $m$ & $\ell\mid m$ \\ \hline
			$G_4=L_2$ & $2^3 \cdot 3$ & $2$ & $L_2$ & $2^3 \cdot 3$ & $2$ \\ \cline{3-5}
			& & $3$ & $L_1$ & $3$ & \\ \hline
			$G_6=C_4\circ L_2$ & $2^4 \cdot 3$ & $2$ & $G_6$ & $2^4 \cdot 3$ & $2$  \\ \cline{3-5}
			& & $3$ & $L_1$ & $3$ & \\ \hline
			\begin{tabular}{@{}c@{}}  $G_k,$ \\ $k=5$ and $7\leq k\leq 22$ \end{tabular} & $|G_k|$ & $\ell \mid |G_k|$ & $G_k$ & $|G_k|$ & $\ell \mid |G_k|$ \\ \hline
			$G_{23}=H_3$ & $2^3 \cdot 3 \cdot 5$ & $2$ & $H_3$ & $2^3 \cdot 3 \cdot 5$ & $2$ \\ \cline{3-5}
			& & $3$ & $A_2$ & $2 \cdot 3$ & \\ \cline{3-5}
			& & $5$ & $D_2^{(5)}$ & $2 \cdot 5$ & \\ \hline
			$G_{24}=J_3^{(4)}$ & $2^4 \cdot 3 \cdot 7$ & $2, 7$ & $J_3^{(4)}$ & $2^4 \cdot 3 \cdot 7$ & $2,7$ \\ \cline{3-5}
			& & $3$ & $A_2$ & $2 \cdot 3$ &  \\ \hline
			$G_{25}=L_3$ & $2^3 \cdot 3^4$ & $2$ & $L_2$ & $2^3 \cdot 3$ & $3$ \\ \cline{3-5}
			& & $3$ & $L_3$ & $2^3 \cdot 3^4$ & \\ \hline
			$G_{26}=M_3$ & $2^4 \cdot 3^4$ & $2, 3$ & $M_3$ & $2^4 \cdot 3^4$ & $2,3$ \\ \hline
			$G_{27}=J_3^{(5)}$ & $2^4 \cdot 3^3 \cdot 5$ & $2, 3$ & $J_3^{(5)}$ & $2^4 \cdot 3^3\cdot 5$ & $2,3$ \\ \cline{3-5}
			& & $5$ & $D_2^{(5)}$ & $2 \cdot 5$ & \\ \hline
			$G_{28}=F_4$ & $2^{7} \cdot 3^2$ & $2, 3$ & $F_4$ & $2^{7} \cdot 3^2$ & $2,3$ \\ \hline
			$G_{29}=N_4$ & $2^9 \cdot 3 \cdot 5$ & $2, 5$ & $N_4$ & $2^9 \cdot 3 \cdot 5$ & $2,5$ \\ \cline{3-5}
			& & $3$ & $A_2$ & $2\cdot 3$ & \\ \hline
			$G_{30}=H_4$ & $2^6 \cdot 3^2 \cdot 5^2$ & $2, 3, 5$ & $H_4$ & $2^6 \cdot 3^2 \cdot 5^2$ & $2, 3, 5$ \\ \hline
			$G_{31}=O_4$ & $2^{10} \cdot 3^2 \cdot 5$ & $2, 3, 5$ & $O_4$ & $2^{10} \cdot 3^2 \cdot 5$ & $2, 3, 5$ \\ \hline
			$G_{32}=L_4$ & $2^{7} \cdot 3^5 \cdot 5$ & $2, 3, 5$ & $L_4$ & $2^{7} \cdot 3^5 \cdot 5$ & $2, 3, 5$ \\ \hline
			$G_{33}=K_5$ & $2^{7} \cdot 3^4 \cdot 5$ & $2$ & $K_5$ & $2^{7} \cdot 3^4 \cdot 5$ & $2$ \\ \cline{3-5}
			& & $3$ & $D_4^{(3)}$ & $2^{3} \cdot 3^4$ & \\ \cline{3-5}
			& & $5$ & $A_4$ & $2^3\cdot 3\cdot 5$ & \\ \hline
			$G_{34}=K_6$ & $2^{9} \cdot 3^7 \cdot 5\cdot 7$ & $2, 3, 7$ & $K_6$ & $2^{9} \cdot 3^7 \cdot 5\cdot 7$ & $2,3,7$ \\ \cline{3-5}
			& & $5$ & $A_4$ & $2^3\cdot 3\cdot 5$ & \\ \hline
			$G_{35}=E_6$ & $2^7 \cdot 3^4 \cdot 5$ & $2$ & $D_5$ & $2^7\cdot 3\cdot5$ & $3$ \\ \cline{3-5}
			& & $3$ & $E_6$ & $2^7 \cdot 3^4 \cdot 5$ & \\ \cline{3-5}
			& & $5$ & $A_4$  & $2^3\cdot 3\cdot 5$ & \\ \hline
			$G_{36}=E_7$ & $2^{10} \cdot 3^4 \cdot 5 \cdot 7$ & $2$ & $E_7$ & $2^{10}\cdot 3^4\cdot5\cdot 7$ & $2$ \\ \cline{3-5}
			& & $3$ & $E_6$ & $2^7 \cdot 3^4 \cdot 5$ & \\ \cline{3-5}
			& & $5$ & $A_4$ & $2^3\cdot 3\cdot 5$ & \\ \cline{3-5}
			& & $7$ & $A_6$ & $2^4\cdot 3^2\cdot 5\cdot 7$ & \\ \hline
			$G_{37}=E_8$ & $2^{14} \cdot 3^5 \cdot 5^2 \cdot 7$ & $2, 3, 5$ & $E_8$ & $2^{14} \cdot 3^5 \cdot 5^2 \cdot 7$ & $2,3,5$ \\ \cline{3-5}
			& & $7$ & $A_6$ & $2^4\cdot 3^2\cdot 5\cdot 7$ & \\ \hline
	\end{tabular}}
	\label{table:parabolic}
\end{table}

\newpage

\begin{table}[H]
	\caption{Cuspidal cases of $G$}
	\centering
		\begin{tabular}{ | c | c | c | }
			\hline
			{$G$} & Cuspidal $\ell$ & $| G |$  \\ \hline \hline
			$G(1,1,\ell^i)$ for $i\in\Z^+$ & $\ell$  & $(\ell^i)!$ \\ \hline
			$G(m,p,n)$ for $\ell\mid m$ & $\ell$ & $\frac{m^nn!}{p}$ \\ \hline
			$G(m,1,1)$ & $\ell$  & $m$ \\ \hline
			$G_4=L_2$ & $2$ & $2^3\cdot3$ \\ \hline
			$G_6$ & $2$ & $2^4\cdot3$ \\ \hline
			$G_k$, $k=5$ and $7\leq k \leq 22$ & $\ell \mid |G_k|$ & $|G_k|$ \\ \hline
			$G_{23}=H_3$ & $2$ & $2^3\cdot3\cdot 5$ \\ \hline
			$G_{24}=J_3^{(4)}$ & $7$ & $2^4\cdot3\cdot 7$ \\ \hline
			$G_{25}=L_3$ & $3$ & $2^3\cdot 3^4$ \\ \hline
			$G_{26}=M_3$ & $2,3$ & $2^4\cdot3^4$ \\ \hline
			$G_{27}=J_3^{(5)}$ & $2,3$ & $2^4\cdot3^3\cdot 5$ \\ \hline
			$G_{28}=F_4$ & $2,3$ & $2^7\cdot 3^2$ \\ \hline
			$G_{29}=N_4$ & $2,5$ & $2^9\cdot 3\cdot 5$ \\ \hline
			$G_{30}=H_4$ & $2,3,5$ & $2^{10}\cdot 3^2\cdot 5$ \\ \hline
			$G_{31}=O_4$ & $2,3,5$ & $2^7\cdot 3^5\cdot 5$ \\ \hline
			$G_{32}=L_4$ & $2,5$ & $2^7\cdot3^5\cdot 5$ \\ \hline
			$G_{33}=K_5$ & $2$ & $2^7\cdot 3^4\cdot 5$ \\ \hline
			$G_{34}=K_6$ & $2,3,7$ & $2^9\cdot 3^7\cdot 5\cdot 7$ \\ \hline
			$G_{35}=E_6$ & $3$ & $2^7\cdot3^4\cdot5$ \\ \hline
			$G_{36}=E_7$ & $2$ & $2^{10}\cdot3^4\cdot5\cdot 7$ \\ \hline
			$G_{37}=E_8$ & $2,3,5$ & $2^{14}\cdot3^5\cdot5^2\cdot 7$ \\ \hline
	\end{tabular}
	\label{table:cuspidalcases}
\end{table}

\begin{table}[H]
	\caption{Type of $R_\ell$ in $G$ excluding primitive G of rank $2$}
	\centering
	\scalebox{0.55}{
		\begin{tabular}{ | c | c | c | c | c |}
			\hline
			{$G$} & $|G|$ & $\ell$ & $R_\ell$ & $|R_\ell|$ \\ \hline \hline
			\begin{tabular}{@{}c@{}}  $G_1=G(1,1,n),$ \\ $n\geq 2$ \end{tabular} & $n!$ & $\ell \mid n!$ & $\displaystyle\prod_{i=1}^{k} G(1,1,\ell^i)^{b_i}$ & $\displaystyle\prod_{i=1}^{k} (\ell^i!)^{b_i}$ \\ \hline
			\begin{tabular}{@{}c@{}} $G_2=G(m,p,n),$  \\ $m>1, n>1$ \end{tabular} & $\frac{m^nn!}{p}$  & $\ell \mid \frac{m^nn!}{p}$  & $\left\{ \begin{array}{ll}
			\hat{\theta}_\alpha G(\ell^{\nu_\ell(m)},\ell^{\nu_\ell(p)}, n)\hat{\theta}_\alpha^{-1}, &  \\
			\text{where} \ \alpha\in\{e^{\frac{i\pi k}{m}} \  | \ 0 \leq k\leq \text{gcd}(\frac{p}{\ell^{\nu_\ell(p)}},n)-1 \} & \text{for} \  \ell \mid p \\
			\displaystyle\prod_{i=0}^{k}G(\ell^{\nu_\ell(m)},1, \ell^i)^{b_i} & \text{for} \ \ell\mid m \ \text{and} \ \ell \nmid p, \\
			\displaystyle\prod_{i=1}^{k} G(1,1,\ell^i)^{b_i} & \text{for} \ \ell \nmid m
			\end{array} \right.$ & $\left\{ \begin{array}{ll}
			\ell^{n\nu_\ell(m)-\nu_\ell(p)}n! & \text{for} \  \ell\mid p, \\
			\displaystyle\prod_{i=0}^{k} [\ell^{\ell^i\nu_\ell(m)}(\ell^i!)]^{b_i} & \text{for} \ \ell\mid m \ \text{and} \ \ell \nmid p, \\
			\displaystyle\prod_{i=1}^{k} (\ell^i!)^{b_i} & \text{for} \ \ell \nmid m
			\end{array} \right.$ \\ \hline
			\begin{tabular}{@{}c@{}}  $G_3=G(m,1,1),$ \\ $m>1$ \end{tabular} & $m$ & $\ell \mid m$ & $G(\ell^{\nu_\ell(m)},1,1)$ & $\ell^{\nu_\ell(m)}$ \\ \hline
			$G_{23}=H_3$ & $2^3 \cdot 3 \cdot 5$ & $2$ & $A_1^3$ & $2^3$ \\ \cline{3-5}
			& & $3$ & $A_2$ & $2 \cdot 3$ \\ \cline{3-5}
			& & $5$ & $D_2^{(5)}$ & $2 \cdot 5$ \\ \hline
			$G_{24}=J_3^{(4)}$ & $2^4 \cdot 3 \cdot 7$ & $2$ & $A_1\times B_2$ & $2^4$  \\ \cline{3-5} & & $3$ & $A_2$ & $2 \cdot 3$ \\ \cline{3-5}
			& & $7$ & $J_3^{(4)}$ & $2^4 \cdot 3 \cdot 7$  \\ \hline
			$G_{25}=L_3$ & $2^3 \cdot 3^4$ & $2$ & $L_2$ & $2^3 \cdot 3$  \\ \cline{3-5}
			& & $3$ & $L_3$ & $2^3 \cdot 3^4$ \\ \hline
			$G_{26}=M_3$ & $2^4 \cdot 3^4$ & $2$ & $A_1\times L_2$ & $2^4\times 3$ \\ \cline{3-5}
			& & $3$ & $L_3 \ \text{and} \ B_3^{(3)}$ & $2^3 \cdot 3^4 \ \text{and} \  2 \cdot 3^4$ \\ \hline
			$G_{27}=J_3^{(5)}$ & $2^4 \cdot 3^3 \cdot 5$ & $2$ & $A_1\times B_2$ & $2^4$ \\ \cline{3-5}
			& & $3$ & $D_3^{(3)}$ & $2^3 \cdot 3^3$ \\ \cline{3-5}
			& & $5$ & $D_2^{(5)}$ & $2 \cdot 5$ \\ \hline
			$G_{28}=F_4$ & $2^{7} \cdot 3^2$ & $2$ & $B_4 \ \text{and} \ \tilde{B}_4$ & $2^{7} \cdot 3$ \\ \cline{3-5}
			& & $3$ & $A_2^2$ & $2^2 \cdot 3^2$ \\ \hline
			$G_{29}=N_4$ & $2^9 \cdot 3 \cdot 5$ & $2$ & $D_4^{(4)}$ & $2^9 \cdot 3$ \\ \cline{3-5}
			& & $3$ & $A_2$ & $2\cdot 3$ \\ \cline{3-5}
			& & $5$ & $A_4$ and $\tilde{A}_4$ & $2^3\cdot 3 \cdot 5$  \\ \hline
			$G_{30}=H_4$ & $2^6 \cdot 3^2 \cdot 5^2$ & $2$ & $D_4$ & $2^6 \cdot 3$ \\ \cline{3-5}
			& & $3$ & $A_2^2$ & $2^2\cdot 3^2$ \\ \cline{3-5}
			& & $5$ & $[D_2^{(5)}]^2$ & $2^2\cdot5^2$  \\ \hline
			$G_{31}=O_4$ & $2^{10} \cdot 3^2 \cdot 5$ & $2$ & $B_4^{(4)}$ & $2^{10} \cdot 3$ \\ \cline{3-5}
			& & $3$ & $A_2^2$ & $2^2\cdot 3^2$ \\ \cline{3-5}
			& & $5$ & $A_4$ and $\tilde{A}_4$ & $2^3\cdot3 \cdot 5$  \\ \hline
			$G_{32}=L_4$ & $2^{7} \cdot 3^5 \cdot 5$ & $2, 5$ & $L_4$ & $2^{7} \cdot 3^5 \cdot 5$  \\ \cline{3-5}
			& & $3$ & $L_1\times L_3$ & $2^3\cdot 3^5$ \\ \hline
			$G_{33}=K_5$ & $2^{7} \cdot 3^4 \cdot 5$ & $2$ & $A_1\times D_4$ & $2^{7} \cdot 3	$ \\ \cline{3-5}
			& & $3$ & $D_4^{(3)}$ & $2^{3} \cdot 3^4$ \\ \cline{3-5}
			& & $5$ & $A_4$ & $2^3\cdot 3\cdot 5$ \\ \hline
			$G_{34}=K_6$ & $2^{9} \cdot 3^7 \cdot 5\cdot 7$ & $2$ & $D_6$ & $2^{9} \cdot 3^2 \cdot 5$ \\ \cline{3-5}
			& & $3$ & $D_6^{(3)}$ & $2^4\cdot 3^7\cdot 5$ \\ \cline{3-5}
			& & $5$ & $A_4$ & $2^3\cdot 3\cdot 5$ \\ \cline{3-5}
			& & $7$ & $A_6$ and $\tilde{A}_6$ & $2^4\cdot 3^2\cdot 5 \cdot 7$ \\ \hline
			$G_{35}=E_6$ & $2^7 \cdot 3^4 \cdot 5$ & $2$ & $D_5$ & $2^7\cdot 3\cdot5$ \\ \cline{3-5}
			& & $3$ & $E_6$ & $2^7 \cdot 3^4 \cdot 5$ \\ \cline{3-5}
			& & $5$ & $A_4$  & $2^3\cdot 3\cdot 5$ \\ \hline
			$G_{36}=E_7$ & $2^{10} \cdot 3^4 \cdot 5 \cdot 7$ & $2$ & $A_1\times D_6$ & $2^{10}\cdot 3^2\cdot5$ \\ \cline{3-5}
			& & $3$ & $E_6$ & $2^7 \cdot 3^4 \cdot 5$  \\ \cline{3-5}
			& & $5$ & $A_4$ & $2^3\cdot 3\cdot 5$ \\ \cline{3-5}
			& & $7$ & $A_6$ & $2^4\cdot 3^2\cdot 5\cdot 7$ \\ \hline
			$G_{37}=E_8$ & $2^{14} \cdot 3^5 \cdot 5^2 \cdot 7$ & $2$ & $D_8$ & $2^{14} \cdot 3^2 \cdot 5 \cdot 7$ \\ \cline{3-5}
			& & $3$ & $A_2\times E_6$ & $2^8\cdot 3^5\cdot 5$ \\ \cline{3-5}
			& & $5$ & $A_4^2$ & $2^6\cdot 3^2\cdot 5^2$ \\ \cline{3-5}
			& & $7$ & $A_6$ & $2^4\cdot 3^2\cdot 5\cdot 7$ \\ \hline
	\end{tabular}}
	\label{table:reflection}
\end{table}

\begin{table}[H]
	\caption{Type of $R_\ell$ in primitive $G$ of rank $2$}
	\centering
	\scalebox{0.9}{
		\begin{tabular}{ | c | c | c | c | c |}
			\hline
			{$G$} & $|G|$ & $\ell$ & $R_\ell$ & $|R_\ell|$ \\ \hline \hline
			$G_4=L_2$ & $2^3 \cdot 3$ & $2$ & $L_2$ & $2^3 \cdot 3$   \\ \cline{3-5}
			& & $3$ & $L_1$ & $3$ \\ \hline
			$G_{5}=\mathcal{C}_3\times \mathcal{T}$ & $2^3 \cdot 3^2$ & $2$ & $L_2$ and $\tilde{L}_2$ & $2^3 \cdot 3$ \\ \cline{3-5}
			& & $3$ & $L_1^2$ & $3^2$ \\ \hline
			$G_{6}=\mathcal{C}_4\circ L_2$ & $2^4 \cdot 3$ & $2$ & $B_2^{(4)}$ & $2^4$  \\ \cline{3-5} 
			& & $3$ & $L_1$ & $3$  \\ \hline
			$G_{7}=\mathcal{C}_{3}\times(\mathcal{C}_4\circ\mathcal{T})$ & $2^4 \cdot 3^2$ & $2$ & $B_2^{(4)}$ & $2^4$  \\ \cline{3-5}
			& & $3$ & $L_1^2$ & $3^2$ \\ \hline
			$G_{8}=\mathcal{T}\mathcal{C}_4$ & $2^5 \cdot 3$ & $2$ & $G(4,1,2)$ & $2^5$ \\ \cline{3-5}
			& & $3$ & $G_8$ & $2^5\cdot3$  \\ \hline
			$G_{9}=\mathcal{C}_8\circ\mathcal{O}$ & $2^6 \cdot 3$ & $2$ & $G(8,2,2)$ & $2^6$ \\ \cline{3-5}
			& & $3$ & $A_2$, $\tilde{A_2}$ and $G_{8}$ & $2\cdot 3$, $2\cdot 3$ and $2^5\cdot 3$  \\ \hline
			$G_{10}=\mathcal{C}_{3}\times\mathcal{T}\mathcal{C}_4$ & $2^{5} \cdot 3^2$ & $2$ & $G(4,1,2)$ & $2^{5}$ \\ \cline{3-5}
			& & $3$ & $L_1^2$ & $3^2$ \\ \hline
			$G_{11}=\mathcal{C}_{3}\times(\mathcal{C}_8\circ\mathcal{O})$ & $2^6 \cdot 3^2$ & $2$ & $G(8,2,2)$ & $2^6$ \\ \cline{3-5}
			& & $3$ & $L_1^2$ & $3^2$  \\ \hline
			$G_{12}\cong GL_2(\mathbb{F}_3)$ & $2^4 \cdot 3$ & $2$ & $G_{12}$ & $2^4 \cdot 3$  \\ \cline{3-5}
			& & $3$ & $A_2$ and $\tilde{A_2}$ & $2 \cdot 3$ \\ \hline
			$G_{13}=\mathcal{C}_4\circ\mathcal{O}$ & $2^{5} \cdot 3$ & $2$ & $G(8,4,2)$ & $2^{5}$ \\ \cline{3-5}
			& & $3$ & $A_2$ and $\tilde{A_2}$ & $2\cdot 3$ \\ \hline
			$G_{14}=\mathcal{C}_3\times G_{12}$ & $2^{4} \cdot 3^2$ & $2$ & $G_{12}$ & $2^{4} \cdot 3$  \\ \cline{3-5}
			& & $3$ & $L_1^2$& $3^2$ \\ \hline
			$G_{15}=\mathcal{C}_{3}\times(\mathcal{C}_4\circ\mathcal{O})$ & $2^{5} \cdot 3^2$ & $2$ & $G(8,4,2)$ & $2^{5}$ \\ \cline{3-5}
			& & $3$ & $L_1^2$ & $3^2$  \\ \hline
			$G_{16}=\mathcal{C}_5\times\mathcal{I}$ & $2^{3} \cdot 3 \cdot 5^2$ & $2,3$ & $G_{16}$ & $2^{3} \cdot 3 \cdot 5^2$ \\ \cline{3-5}
			& & $5$ & $G(5,1,1)^2$ & $5^2$  \\ \hline
			$G_{17}=\mathcal{C}_{5}\times(\mathcal{C}_4\circ\mathcal{I})$ & $2^4 \cdot 3 \cdot 5^2$ & $2$ & $B_2^{(4)}$ & $2^4$ \\ \cline{3-5}
			& & $3$ & $A_2$, $\tilde{A_2}$ and $G_{16}$ & $2\cdot 3$, $2\cdot 3$ and $2^{3} \cdot 3 \cdot 5^2$ \\ \cline{3-5}
			& & $5$ & $G(5,1,1)^2$  & $5^2$ \\ \hline
			$G_{18}=\mathcal{C}_{15}\times\mathcal{I}$ & $2^{3} \cdot 3^2 \cdot 5^2$ & $2$ & $L_2$, $\tilde{L_2}$ and $G_{16}$ & $2^3\cdot 3$, $2^3\cdot 3$ and $2^3\cdot 3\cdot 5^2$  \\ \cline{3-5}
			& & $3$ & $L_1^2$ & $3^2$  \\ \cline{3-5}
			& & $5$ & $G(5,1,1)^2$ & $5^2$ \\ \hline
			$G_{19}=\mathcal{C}_{15}\times(\mathcal{C}_4\circ\mathcal{I})$ & $2^{4} \cdot 3^2 \cdot 5^2$ & $2$ & $B_2^{(4)}$ & $2^{4}$ \\ \cline{3-5}
			& & $3$ & $L_1^2$ & $3^2$ \\ \cline{3-5}
			& & $5$ & $G(5,1,1)^2$ & $5^2$ \\ \hline
			$G_{20}=\mathcal{C}_3\times\mathcal{I}$ & $2^{3} \cdot 3^2 \cdot 5$ & $2$ & $L_2$ and $\tilde{L_2}$ & $2^{3} \cdot 3$ \\ \cline{3-5}
			& & $3$ & $L_1^2$ & $3^2$ \\ \cline{3-5}
			& & $5$ & $G_{20}$ & $2^{3} \cdot 3^2 \cdot 5$ \\ \hline
			$G_{21}=\mathcal{C}_{3}\times(\mathcal{C}_4\circ\mathcal{I})$ & $2^{4} \cdot 3^2 \cdot 5$ & $2$ & $B_2^{(4)}$ & $2^{4}$ \\ \cline{3-5}
			& & $3$ & $L_1^2$ & $3^2$ \\ \cline{3-5}
			& & $5$ & $D_2^{(5)}$, $\tilde{D_2^{(5)}}$ and $G_{20}$ & $2\cdot 5$, $2\cdot 5$ and $2^{3} \cdot 3^2 \cdot 5$ \\ \hline
			$G_{22}=\mathcal{C}_4\circ\mathcal{I}$ & $2^{4} \cdot 3 \cdot 5$ & $2$ & $B_2^{(4)}$ & $2^{4}$ \\ \cline{3-5}
			& & $3$ & $A_2$ and $\tilde{A_2}$ & $2\cdot3$ \\ \cline{3-5}
			& & $5$ & $D_2^{(5)}$ and $\tilde{D_2^{(5)}}$ & $2\cdot5$ \\ \hline
	\end{tabular}}
	\label{table:reflectionrank2}
\end{table}

\begin{table}[H]
	\caption{Supercuspidal cases of $G$}
	\centering
		\begin{tabular}{ | c | c | c | }
			\hline
			{$G$} & $\ell$-Sylow prime & $\mid G \mid$  \\ \hline \hline
			$G(1,1,\ell^i)$ for $i\in\Z^+$ & $\ell$  & $(\ell^i)!$ \\ \hline
			$G(\ell^i,\ell^j,n)$ for $j\leq i \in\Z^+, n>1$ & $\ell$  & $\ell^{in-j}n!$ \\ \hline
			$G(\ell^i,1,\ell^j)$ for $i,j\in\Z^+$ & $\ell$  & $\ell^{i\ell^j}(\ell^j)!$ \\ \hline
			$G(\ell^i,1,1)$ for $i\in\Z^+$ & $\ell$  & $\ell^i$ \\ \hline
			$G_4=L_2$ & $2$ & $2^3\cdot3$ \\ \hline
			$G_8$ & $3$ & $2^5\cdot 3$ \\ \hline
			$G_{12}$ & $2$ & $2^4\cdot 3$  \\ \hline
			$G_{16}$ & $2,3$ & $2^3\cdot 3\cdot 5^2$ \\ \hline
			$G_{20}$ & $5$ & $2^3\cdot 3^2\cdot 5$ \\ \hline
			$G_{24}=J_3^{(4)}$ & $7$ & $2^4\cdot3\cdot 7$ \\ \hline
			$G_{25}=L_3$ & $3$ & $2^3\cdot 3^4$ \\ \hline
			$G_{32}=L_4$ & $2,5$ & $2^7\cdot3^5\cdot 5$ \\ \hline
			$G_{35}=E_6$ & $3$ & $2^7\cdot3^4\cdot5$ \\ \hline
	\end{tabular}
	\label{table:sylowcases}
\end{table}

\begin{table}[H]
	\caption{Cases where $R_\ell$ is not unique up to conjugacy}
	\centering
		\begin{tabular}{ | c | c | c | c | c | c | }
			\hline
			{$G$} & $\ell$ & $R_\ell$ & $|R_\ell|$ & Conjugacy Classes \\ \hline \hline
			$G_2=G(m,p,n)$ & $\ell \mid p$  & $G(\ell^{\nu_\ell(m)},\ell^{\nu_\ell(p)}, n)$ & $\ell^{n\nu_\ell(m)-\nu_\ell(p)}n!$  & $\text{gcd}(\frac{p}{\ell^{\nu_\ell(p)}},n)$ \\ \hline
			$G_5$ & $2$ & $L_2$ & $2^3\cdot3$ & $2$ \\ \hline
			$G_9$ & $3$ & $A_2$ & $2\cdot3$ & $2$ \\
			\cline{3-5}
			& & $G_{8}$ & $2^5\cdot3$ & $1$ \\ \hline
			$G_{12}$ & $3$ & $A_2$ & $2\cdot3$ & $2$ \\ \hline
			$G_{13}$ & $3$ & $A_2$ & $2\cdot3$ & $2$ \\ \hline
			$G_{17}$ & $3$ & $A_2$ & $2\cdot3$ & $2$ \\
			\cline{3-5}
			& & $G_{16}$ & $2^3\cdot3\cdot5^2$ & $1$ \\ \hline
			$G_{18}$ & $2$ & $L_2$ & $2^3\cdot3$ & $2$ \\ \hline
			$G_{20}$ & $2$ & $L_2$ & $2^3\cdot3$ & $2$ \\ \hline
			$G_{21}$ & $5$ & $D_2^{(5)}$ & $2\cdot5$ & $2$ \\ \cline{3-5}
			& & $G_{20}$ & $2^3\cdot3^2\cdot5$ & $1$ \\ \hline
			$G_{22}$ & $3$ & $A_2$ & $2\cdot3$ & $2$ \\ \cline{2-5}
			& $5$ & $D_2^{(5)}$ & $2\cdot5$ & $2$ \\ \hline
			$G_{26}=M_3$ & $3$ & $L_3$ & $2^3 \cdot 3^4$ & $1$ \\ \cline{3-5}
			&  & $B_3^{(3)}$ & $2 \cdot 3^4$ & $1$ \\ \hline
			$G_{28}=F_4$ & $2$ & $B_4$ & $2^7\cdot 3$ & $2$ \\ \hline
			$G_{29}=N_4$ & $5$ & $A_4$ & $2^3\cdot 3\cdot 5$ & $2$ \\  \hline
			$G_{31}=O_4$ & $5$ & $A_4$ & $2^3 \cdot 3 \cdot 5$ & $2$ \\ \hline
			$G_{34}=K_6$ & $7$ & $A_6$ & $2^{4} \cdot 3^2 \cdot 5\cdot 7$ & $2$ \\ \hline
	\end{tabular}
	\label{table:notunique}
\end{table}
\newpage
\bibliographystyle{alpha}
\bibliography{references.bib}
\end{document}